\documentclass[12pt]{amsart}
\usepackage[english]{babel}
\usepackage{amssymb, amsthm, amsmath}
\usepackage{txfonts}
\usepackage{calrsfs}
\usepackage{mathrsfs}
\usepackage{amscd}
\usepackage{graphicx}
\usepackage[all]{xy} 
\usepackage[applemac]{inputenc} 
\usepackage[T1]{fontenc} 
\usepackage{hyperref} 
\usepackage{cite}

\DeclareMathOperator*{\colim}{colim} 

\DeclareMathOperator*{\Tot}{Tot}

\begin{document}

\pretolerance=3000 

\bibliographystyle{./hsiam} 

\newcommand{\cat}[1]{\mathscr{#1}}
\newcommand{\ob}{\textrm{Ob}}
\newcommand{\mor}{\textrm{Mor}}
\newcommand{\id}{\mathbf 1} 

\newcommand{\sSet}{\mathbf{sSet}}
\newcommand{\Alg}{\mathbf{Alg}}
\newcommand{\sPr}{\mathbf{sPr}}
\newcommand{\dgCat}{\mathbf{dgCat}} 
\newcommand{\dgAlg}{\mathbf{dgAlg}} 
\newcommand{\Vect}{\mathbf{Vect}} 
\newcommand{\Ch}{\mathbf{Ch}}
\newcommand{\Chd}{{\mathbf{Ch}_{dg}}} 
\newcommand{\Chdp}{{\mathbf{Ch}_{pe}}} 
\newcommand{\Chp}{{\mathbf{Ch}_{pe}}} 
\newcommand{\uChp}{{\underline {\mathbf{Ch}}{}_{pe}}} 
\newcommand{\skMod}{\mathbf{skMod}} 
\newcommand{\Cat}{\mathbf{Cat}} 
\newcommand{\sCat}{\mathbf{sCat}}
\newcommand{\sModCat}{\mathbf{sModCat}}

\newcommand{\Hom}{\mbox{Hom}}
\newcommand{\uHom}{\mbox{\underline{Hom}}}
\newcommand{\Aut}{\mbox{Aut}}
\newcommand{\Out}{\mbox{Out}}
\newcommand{\End}{\mbox{End}}
\newcommand{\Map}{\mbox{Map}}
\newcommand{\map}{\mbox{map}}
\newcommand{\Tor}{\mbox{Tor}}
\newcommand{\Ext}{\mbox{Ext}}

\newcommand{\mods}{\textrm{-Mod}}
\newcommand{\Kos}{\mbox{Kos}} 
\newcommand{\Spec}{\mbox{Spec}}

\newcommand{\set}[1]{\mathbb{#1}}
\newcommand{\Q}{\mathbb{Q}}
\newcommand{\C}{\mathbb{C}}
\newcommand{\Z}{\mathbb{Z}}
\newcommand{\R}{\mathbb{R}}

\newcommand{\De}{\Delta}
\newcommand{\Ga}{\Gamma}
\newcommand{\Om}{\Omega}
\newcommand{\ep}{\epsilon}
\newcommand{\de}{\delta}
\newcommand{\La}{\Lambda}
\newcommand{\la}{\lambda}
\newcommand{\al}{\alpha}
\newcommand{\om}{\omega}

\newcommand{\oo}{\infty}

\newcommand{\op}{^{\textrm{op}}}
\newcommand{\inv}{^{-1}} 
\newcommand{\nneg}{\tau_{\geq 0}} 

\newcommand{\Bold}{\boldsymbol}
\newcommand{\IF}{\textrm{if }}
\newcommand{\Res}{\textrm{Res}}
\newcommand{\comment}[1]{}

\theoremstyle{plain}
\newtheorem{theorem}{Theorem}[section]
\newtheorem{lemma}[theorem]{Lemma}
\newtheorem{proposition}[theorem]{Proposition}

\theoremstyle{definition}
\newtheorem{notation}[theorem]{Notation}
\newtheorem{definition}[theorem]{Definition}
 
\theoremstyle{remark}
\newtheorem{remark}[theorem]{Remark}
\newtheorem{example}[theorem]{example}


\title[Properness and simplical resolutions for $\dgCat$]{Properness and simplicial resolutions \\ for the model category $\dgCat$}
\author{Julian V. S. Holstein}
\address{Julian Holstein \\ Universit\"at Hamburg \\ Fachbereich Mathematik \\ Bundesstra{\ss}e 55 \\ 20146 Hamburg}
\email{julian.holstein1@uni-hamburg.de}

\begin{abstract}
We give an elementary proof that the model category of dg-categories over a ring of flat dimension 0 is left proper and we provide a construction of simplicial resolutions in dg-categories, given by categories of Maurer-Cartan elements.
\end{abstract}

\maketitle
\section{Introduction}
We provide proofs of the following properties of the model category $\dgCat_{k}$ of dg-categories (with the Morita or Dwyer-Kan model structure) over a ring $k$. 
\begin{itemize}
\item When $k$ has flat dimension 0, the category $\dgCat_k$ is left proper. 
\item Natural simplicial resolutions in $\dgCat$ are given by dg-categories of Maurer-Cartan elements.
\end{itemize}

Left properness is essential to show the existence of Bousfield localizations of dg-categories. (Under stronger assumptions on $k$ left properness also follows from \cite{Muro12}.)
We also remark that $\dgCat$ is cellular and there is a Quillen equivalent combinatorial subcategory (without assumptions on the existence of large cardinals).

Simplicial resolutions allow for constructions of explicit mapping spaces and simplicial actions. These play a crucial role in categorifying cohomology to Morita cohomology, see \cite{Holstein1}. 
We construct simplicial resolutions by an explicit if somewhat lengthy computation motivated by the \v Cech globalization in \cite{Simpson05}.
Note that the explicit combinatorics of this construction have appeared in other contexts: 
If $K$ is the nerve of a category this is the data of an $A_{\oo}$-functor, see for example \cite{Igusa02}.
If $K$ is any simplicial set one recovers the $\oo$-local systems defined in \cite{Block09}.
We feel that the interpretation here as the cotensor action of simplicial sets on $\dgCat$, computed via simplicial resolutions, provides a satisfying conceptual viewpoint.

These results are taken from the author's thesis. 
Thanks are due to Ian Grojnowski and Jon Pridham for helpful discussions as well as to Zhaoting Wei and the anonymous referee for useful questions, corrections and suggestions. 

Finally, the author is grateful to Daria Poliakova for pointing out some gaps in the proofs. Some of those gaps are closed in \cite{Arkhipov18}.
We indicate the problems in footnotes and explain how to correct them in a new appendix. 
All results remain correct.
\subsection{Conventions}
We assume the reader is familiar with the theory of dg-categories. Basic references are \cite{Keller06} and \cite{Toen07a}. 

Recall in particular that there are two model structures on $\dgCat_{k}$, the category of differential graded categories over a ring $k$. These are the Dwyer-Kan model structure, constructed in \cite{Tabuada04}, and the Morita model structure \cite{Tabuada05} which is its left Bousfield localization, cf. \cite{Tabuada07}. We will often not distinguish between them as our results will apply to both model categories.\footnote{While the results are indeed true for both model structures, a careful distinction is needed for correct proofs, see the appendix.} 

We use homological grading conventions, all differentials decrease the degree. The degree is indicated by a subscript or the inverse of a superscript, $C_{i} = C^{-i}$.

\section{Dg-categories over a ring of flat dimension 0 form a left proper, cellular, combinatorial model category}\label{sect-furtherdg}
\subsection{Left properness}
In this section we will show that the model category of dg-categories over a field $k$ is left proper. Recall that a model category is \emph{left proper} if any pushout of a weak equivalence along a cofibration is again a weak equivalence.

\begin{remark} Recall that $\dgCat$ with the Dwyer-Kan model structure is right proper since every object is fibrant, and it is not right proper with the Morita model structure, as is shown explicitly by Example 4.10 in \cite{Tabuada10a}.
\end{remark}

Before proceeding to the proof we mention two closely related results from the literature. Dwyer and Kan prove left properness for simplicial categories on a fixed set of objects in \cite{Dwyer80}.

If we strengthen our assumption and let $k$ have global dimension 0, then it follows from Corollary 1.3 in \cite{Muro12} that $\dgCat_{k}$ is left proper.
To see this, note that in this case all chain complexes over $k$ are cofibrant in the projective model structure, so the results in \cite{Muro12} apply. Indeed, any chain complex is a direct limit of its canonical filtration by bounded below subcomplexes. If all $k$-modules are projective this is a special direct limit in the sense of \cite{Spaltenstein88}, hence the limit is a K-projective object and hence cofibrant.

\begin{theorem}\label{thm-dgcat-leftproper}
If $k$ has flat dimension 0 the model category $\dgCat_{k}$ is left proper.
\end{theorem}
\begin{proof}
Left Bousfield localization preserves left properness, see Proposition 3.4.4 of \cite{Hirschhorn03},
so it is enough to show $\dgCat$ with the Dwyer-Kan model structure is left proper.

The main work is in showing that pushout along the generating cofibrations preserves quasi-equivalences.

To see this suffices note first that transfinite compositions are just filtered colimits, 
and filtered colimits preserve quasi-equivalences as follows:
A filtered colimit of categories can be computed set-theoretically on objects and morphisms.
Now filtered colimits preserve weak equivalences of simplicial sets 
and hence of mapping spaces. They also preserve the homotopy category since a filtered colimit of equivalences of categories is an equivalence of categories and taking the homotopy category commutes with filtered colimits. 
Second, if pushout along some map preserves weak equivalences then so does pushout along a retract by functoriality of colimits. Since all cofibrations are retracts of transfinite compositions of generating cofibrations, it does indeed suffice to check generating cofibrations. 

Recall the generating cofibrations of $\dgCat$ \cite{Tabuada10a}. 
We write $k$ for the dg-category with one object with endomorphisms $k$ concentrated in degree 0. Also let $\cat S(n-1)$ have two objects $a$ and $b$ and $\End(a) \cong \End(b) \cong k[0]$ while $\Hom(a, b) = k.g$ with $g$ in degree $n-1$ and $\Hom(b,a) = 0$. Finally let $\cat D(n)$ be obtained by $\cat S(n-1)$ by adding a generating morphism $f$ of degree $n$ to $\cat S(n-1)$ with $df = g$.
Then the generating cofibrations of $\dgCat$ are given by $\emptyset \to k$ and by
$\cat S(n-1) \to \cat D(n)$ for all $n \in \set Z$.

It is clear that pushout along $\emptyset \to k$ preserves quasi-equivalences.

So consider the generating cofibration $\cat S(n-1) \to \cat D(n)$ with a map $j\colon \cat S(n-1) \to \cat C$ 
and a quasi-equivalence $F\colon \cat C \to \cat E$. 
In forming the pushforward we adjoin a new map $f$ with $df = j(g)$. We call the resulting category $\cat C'$. Then let $\cat E'$ be the pushout of $\cat S(n-1) \to \cat D(n)$ along $F \circ j$.

The pushout along $j$ has the same objects as $\cat C$.
The morphism space is obtained by collecting maps from $C$ to $D$, graded by how often they factor through $f \colon j(a) \to j(b)$. Write  $\cat C(A, B)$ etc.\ for the enriched hom-spaces $\uHom_{\cat C}(A, B)$ etc.
Then the hom-spaces in $\cat C'$ are given as follows:
\begin{equation}\label{eqn-barhom}
 \cat C'(C, D) = \Tot{\!}^{\oplus} \ \big( \cat C(C, D) \oplus \big(\cat C(j(b), D) \otimes k.f \otimes T \otimes \cat C(C, j(a)) \big) \big)
 \end{equation}
Here $T = \sum_{n \geq 0} (\cat C(j(b), j(a)) \otimes k.f)^{\otimes n}$ and we introduce a horizontal degree $n$ with $\cat C(C, D)$ in degree $-1$.
The right hand side has a vertical differential $d_{v}$ given by the internal differential and a horizontal differential $d_{h}$ given by $f \mapsto j(g) \in \Hom(j(b), j(a))$ composed with the necessary compositions. 

If the functor $F$ is not the identity on objects from $\cat C$ to $\cat E$ we factor 
\[
F = Q \circ H\colon \cat C \to \cat D \to \cat E
\]
where $\cat D$ has as objects the objects of $\cat C$ but $\Hom_{\cat D}(A, B) = \Hom_{\cat E}(FA, FB)$. Then $H$ is identity on objects and $Q$ is an isomorphism on hom-spaces. We form the pushforward and obtain the factorization $F' = Q' \circ H'$ through $\cat D'$. 

So it suffices to prove the following two lemmas. \end{proof}

\begin{lemma} The functor $Q'$ defined as above is a quasi-equivalence if $Q$ is.
\end{lemma}
\begin{proof}
Note that $Q'$ is quasi-essentially surjective if $Q$ is since both $\cat D \to \cat D'$ and $\cat E \to \cat E'$ are essentially surjective as pushout along $j$ does not change the set of objects.

Next we use a spectral sequence to compute the hom-spaces in $\cat D'$ and $\cat E'$. To construct the spectral sequence we filter the right-hand side of Equation \ref{eqn-barhom} (with $\cat D$ respectively $\cat E$ in place of $\cat C$) by columns, i.e.\ by $n$. Let $(V, d_{h}+d_{v})$ denote any Hom space in $\cat D'$ or $\cat E'$.
The filtration is bounded below and exhaustive for the direct sum total complex $V$ 
and hence the associated spectral sequence 
\[
E^{1}_{pq} = H_{p+q}(\textrm{Gr}_{p} V) \Rightarrow E^{\oo}_{pq} = H_{p+q}(V)
\]
converges. 
Now $\textrm{Gr}(V) = (V, d_{v})$ and the map induced by $Q'$ is given by $Q$ on all the hom-spaces making up the right-hand side of Equation \ref{eqn-barhom}.
Since $Q$ induces isomorphisms on hom-spaces, it induces isomorphisms on their direct sums and tensor products and thus $Q'$ induces an isomorphisms on the $E^{1}$-page of the spectral sequences computing hom-spaces in $\cat D'$ and $\cat E'$. Hence $Q'$ induces an isomorphism on the $E^{\oo}$-page. For any pair of objects $C,D$ in $\cat D'$ this gives an isomorphism $\cat D'(C, D) \cong \cat E'(QC, QD)$, so $Q'$ induces quasi-isomorphisms on hom-spaces. 

Note that since $\cat S(n-1)$ maps to $\cat E$ via $\cat D$ all the hom-spaces involved in computing $\cat E'(QC, QD)$ are indeed images of hom-spaces in $\cat D$.
\end{proof}
\begin{lemma} The functor $H'$ defined as above is a quasi-equivalence if $H$ is.
\end{lemma}
\begin{proof}
Note that $H'$ is quasi-essentially surjective if $H$ is for the same reason that $Q'$ is. 

To consider the effect of $H'$ on mapping spaces we follow the same argument as in the previous lemma.
Now $H$ only induces weak equivalences on hom-spaces, but we know all hom-spaces are flat over $k$ by assumption. Hence the tensor product in Equation \ref{eqn-barhom} preserves quasi-isomorphisms. So we have a quasi-isomorphism between the $E^{1}$ pages of the spectral sequences and hence between $E^{\oo}$ pages and $H'$ induces quasi-isomorphisms on hom-spaces. 
\end{proof}

\begin{remark}\label{rk-nonflat}
If $k$ does not have flat dimension 0 then the conclusion is false. We can adapt Example 2.7 in \cite{Rezk02} to the case of dg-categories. Let $k$ have positive flat dimension, then there exists a pair of $k$-modules $M,N$ with $\Tor_{1}^{k}(M,N) \neq 0$. We will consider the $k$-algebra $A = k \oplus M \oplus N$ with trivial product $M \oplus N$. Then $\Tor_{1}^{k}(A,A) \neq 0$. View $A$ as a dg-algebra concentrated in degree 0 and take a free resolution $B$ of $A$. Next consider both $A$ and $B$ as dg-categories with one object. They are quasi-equivalent. Now attach a free generator to $A$ and to $B$ by pushout along the generating cofibration $\cat S(-1) \to \cat D(0)$. We then have $A\langle x\rangle \simeq \bigoplus_{n\geq 1} A^{\otimes n}$ and $B\langle y \rangle \simeq \bigoplus_{n\geq 1} B^{n}$ (note that the tensor product here is the underived tensor product over $k$). But since $H_{1}(B\otimes B) = \Tor_{1}^{k}(A, A) \neq 0$ and $A\langle x \rangle$ is concentrated in degree 0 the two pushouts are not quasi-equivalent, and $\dgCat_{k}$ is not left proper.

Hence the model category of dg-categories is left proper if and only if all dg-categories are $k$-flat, i.e.\ if and only if $k$ has flat dimension 0, equivalently if $k$ is von Neumann regular.

In \cite{Rezk02} the existence of a proper model for simplicial $k$-algebras is proven. A similar result for dg-categories is beyond the scope of this work.
\end{remark}

\subsection{Cellularity and combinatoriality}
One of the main uses of properness is in constructing left Bousfield localizations. The only additional assumption needed is that the model category is either cellular or combinatorial. We now show that both are satisfied for $\dgCat$.

\begin{proposition} The two model category structures on $\dgCat$ are cellular. 
\end{proposition}
\begin{proof}
Recall that a model category is cellular if it is cofibrantly generated with generating cofibrations $I$ and generating trivial cofibrations $J$ such that the domains and codomains of the elements of $I$ are compact, the domains of the elements of $J$ are small relative to $I$ and the cofibrations are effective monomorphisms. 
See Chapter 10 of \cite{Hirschhorn03} for more details.

Left Bousfield localization preserves being cellular see Theorem 4.1.1 of \cite{Hirschhorn03}.
So it is enough to show $\dgCat$ with the Dwyer-Kan model structure is cellular.

The domains and codomains of elements of $I$ are categories with at most two objects and perfect hom-spaces, so maps from these objects to relative $I$-complexes factor through small subcomplexes. So domains and codomains of $I$ are compact.

Similarly the domains of the elements of $J$ have two objects and perfect hom-spaces. Hence taking maps from a domain of $J$ commutes with filtered colimits. So domains of $J$ are small relative to $I$.

We are left to check that relative $I$-cell complexes, i.e.\ 
transfinite compositions of pushouts of generating cofibrations, are effective monomorphisms, i.e.\ any relative $I$-cell complex $f\colon X \to Y$ is the equalizer of $Y \rightrightarrows Y \amalg_{X}Y$. 
Note that we form the pushout along a generating cofibration by attaching maps freely. If we form $\cat C'$ and $\cat C''$ from $\cat C$ by attaching maps freely then the equalizer will have the same objects and the hom-spaces are given by considering morphisms of the pushout that are in the image of both $\cat C'$ and $\cat C''$. But these are precisely the hom-spaces of $\cat C$. 
\end{proof}

\begin{definition}
Let $\la$ be a regular cardinal.
An object $A$ in a category $\cat D$ is \emph{$\la$-presentable} if it is small with respect to $\la$-filtered colimits, i.e.\ if for every $\la$-filtered colimit $\colim B_i$ the map $\colim \Hom (A, B_i) \to \Hom(A, \colim B_i)$ is an isomorphism. We say $A$ is \emph{presentable} if it is $\la$-presentable for some $\la$. A cocomplete category is \emph{locally presentable} if for some regular cardinal $\la$ it has a set $S$ of $\la$-presentable objects such that every object is a $\la$-directed colimit of objects in $S$.
\end{definition}

\begin{definition}
A model category is \emph{combinatorial} if the underlying category is locally presentable.
\end{definition}

It is known that there exist combinatorial models for all cofibrantly generated model categories under a large cardinal assumption, cf. \cite{Raptis09}. We notice that this assumption is not necessary for $\dgCat$. 

\begin{proposition}\label{propn-dgcat-combinatorial}
The category $\dgCat$ is Quillen equivalent to a combinatorial subcategory.
\end{proposition}
\begin{proof}
This follows immediately from the proof of the main theorem of \cite{Raptis09}.
Let $\cat D$ denote either of the two model structures on $\dgCat$. Let $S$ be the collection of objects that are domains or codomains of the generating cofibrations and generating trivial cofibrations. (See \cite{Tabuada10a} for an explicit description.) Clearly $S$ is a set. 
Let $\cat S$ denote the full subcategory of $\cat D$ with objects $S$. Define $\eta_{S}(X)$ to be the colimit of the forgetful diagram $(s \to A)  \mapsto s$ indexed by the overcategory $\cat S \downarrow A$.
Then an object $A \in \cat D$ is \emph{$S$-generated} if it is isomorphic to $\eta_S(X)$. 

Now by the proof of Theorem 1.1 in \cite{Raptis09} the subcategory of $S$-generated objects of $\cat D$ is a model category $\cat D_S$ which is Quillen equivalent to the original one. Moreover, by Proposition 3.1 of \cite{Raptis09}, $\cat D_S$ is locally presentable if every object in $S$ is presentable. 
But this is clear since the objects in $S$ have finitely many objects and generating morphisms.
\end{proof}

\begin{remark}
Note that Vop\v enka's principle is not needed here since the objects of $S$ are presentable.
\end{remark}

\section{Simplicial resolutions of dg-categories}\label{section-MCglobal} 

In this section we will construct explicit simplicial functorial resolutions $\cat C \mapsto \cat C_{\bullet}$ in $\dgCat$. Again, we can consider either model structure on $\dgCat$.\footnote{The reader should consider the Dwyer-Kan model structure in this section and refer to the appendix for the proof in the Morita case.}

We first recall the basic definitions.
Let $\De$ be the simplex category and consider the constant diagram functor $c\colon \cat M \to \cat M^{\De\op}$. Then a \emph{simplicial resolution} $M_{\bullet}$ for $M \in \cat M$ is a fibrant replacement for $cM$ in the Reedy model structure on $\cat M^{\De\op}$. (For a definition of the Reedy model structure see for example Chapter 15 of \cite{Hirschhorn03}.) The dual notion is a \emph{cosimplicial} \mbox{\emph{resolution} $M^{\bullet}$}.

We recall two applications:

By using simplicial resolutions one can define mapping spaces with values in $Ho(\sSet)$ for every model category, even if it is not a simplicial model category.
If $cB \to \tilde B$ is a simplicial resolution in $\cat M^{\De\op}$ and $QA$ a cofibrant replacement in $\cat M$ then $\Map(A, B) \simeq \Hom^{\bullet}(QA, \tilde B) \simeq R(\Hom^{\bullet}(-, c-))$, where the right-hand side uses the bifunctor $\Hom^{\bullet}\colon \cat M\op \times \cat M^{\De\op} \to \mathbf{Set}^{\De\op}$ that is defined levelwise.

Moreover, every homotopy category of a model category is tensored and cotensored in $Ho(\sSet)$.
In fact, $\cat M$ can be turned into a simplicial category 
in the sense that there is an enrichment given by the bifunctor $\Map$ and there is a tensor functor as well as a cotensor or power functor, which can be constructed from the simplicial and cosimplicial resolutions. The cotensor is constructed using the simplicial resolution as follows:
Let a simplicial resolution $A_{\bullet} \in \cat M^{\De\op}$ and a simplicial set $K$ be given. Consider $\De K\op$, the opposite of the category of simplices of $K$, with the natural map $v\colon \De K\op \to \De\op$ sending $\De[n] \mapsto K$ to $[n]$.
We define $A_\bullet ^K = \lim_{\De K\op} A_{n}$ to be the image of $A_\bullet$ under $\lim \circ \ v^*\colon \cat C^{\De\op} \to \cat C^{\De K\op} \to \cat C$.
This can also be written as $A^K = \lim_n (\prod_{K_n} A_{n}$).

\subsection{The construction}

Our construction is directly motivated by Simpson's construction of the globalization of a presheaf of dg-categories as a dg-category of Maurer--Cartan elements, cf. section 5.4 of \cite{Simpson05}.

\begin{remark}
In fact, the construction of $\cat C_{n}$ below corresponds to considering the constant presheaf of dg-categories on a covering of $|\De^{n}|$ by $n+1$ open sets (corresponding to leaving out one of the faces).
\end{remark}

\begin{definition}\label{def-construction}
Assume $\cat C$ is fibrant, replace fibrantly otherwise. Then $\cat C_{n}$ is a dg-category with objects given by pairs $(E, \eta)$ where $E$ is a collection $E_0, \dots, E_n \in \ob \cat C$ and $\eta$ is a collection of $\eta_I = \eta(I)\in \uHom_{k-1}(E_{i_0}, E_{i_k})$ for all multi-indices $I = (i_0, \dots, i_k)$ with $1 < k < n$. The case $k = 0$ is subsumed by the differential on $E$. (We interpret $\eta(i) = 0$ where it comes up in computation.) 
These pairs must satisfy the Maurer--Cartan condition: $\de \eta + \eta^2 = 0$, explained below. We also demand that all $\eta_{ij} \in \uHom(E_i, E_j)$ are weak equivalences in $\cat C$.
\end{definition}
\begin{remark}
If we do not fibrantly replace the construction gives a \emph{simplicial framing} on $\dgCat$, see for example \cite{Hirschhorn03}. The simplicial resolution can then be viewed as composing functorial fibrant replacement with the simplicial framing.
\end{remark}

Let us spell out the Maurer--Cartan condition. Intuitively, $\eta$ provides all the comparison maps as well as homotopies between the different compositions. We define the differential 
\[(\de \eta)(i_0, \dots, i_k) \coloneqq d(\eta(i_0, \dots, i_k)) + (-1)^{|\eta|} \sum_{j=1}^{k-1} (-1)^j \eta(i_0, \dots, \widehat{i_j}, \dots, i_k)\]
which lives in $\uHom_{k}(E_{i_0}, E_{i_k})$. We write $\de = d+\De$. Here we define $|\eta| = 1$. 
The product is: 
\[(\phi \circ \eta)(i_0, \dots, i_k) \coloneqq \sum_{j=0}^k  (-1)^{|\phi| j} \phi(i_j, \dots, i_k) \circ \eta(i_0, \dots, i_j)\]
Both definitions follow section 5.2 of \cite{Simpson05}, with some corrections to the signs. We leave out the terms in $\de \eta$ corresponding to leaving out $i_0$ and $i_k$ as they do not live in the correct hom-spaces. 

One can now check that $\De d = - d \De$ (and hence $\de^{2} = 0$) and we have the following Leibniz rule:
\[\de (\phi \circ \eta) = (-1)^{|\eta|}(\de \phi) \circ \eta + \phi \circ (\de \eta)\]
The same equation holds for the summands $d$ and $\De$.
(The unusual sign appears because of the backward notation for compositions.) 

\begin{example} For $n=1$ we have $(\de \eta + \eta^2)_{01} = d(\eta_{01}) + 0$, the expected cycle condition. For $n=2$ we have for example
\[(\de \eta + \eta^2)_{012} = d(\eta_{012}) + \eta_{02} - \eta_{12} \circ \eta_{01} \in \uHom_1(E_0, E_2)\]

So an element of $\cat D_2$ is of the form $(E,\eta)$ where $E = (E_0, E_1, E_2)$ and $\eta =  (\eta_{01}, \eta_{02}, \eta_{12}; \eta_{013})$ satisfies $d\eta + \eta^2 = 0$, which comes out to $d\eta_{ij} = 0$ and $d\eta_{012} =  - \eta_{02} + \eta_{12} \circ \eta_{01}$. This agrees with our intuition that $ \eta_{012}$ is a homotopy from $\eta_{12} \circ \eta_{01}$ to $\eta_{02}$.
\end{example}
Morphisms from $(E, \eta)$ to $(F, \phi)$ are given as follows.
\[\uHom^{-m}_{\cat C_n}((E,\eta), (F,\phi)) = \{a(i_0, \dots, i_k)\}\]
where $a(i_0, \dots, i_k) \in \uHom_{m-k}(E_{i_0}, F_{i_k})$. 
We write $m = |a|$ for the degree of a morphism.
We have a differential $d_{\eta, \phi}$ defined by
\[(d_{\eta, \phi}(a))(i_0, \dots, i_k) = \de(a) + \phi \circ a - (-1)^{|a|} a \circ \eta\]
where composition and differential are defined as above. The Maurer--Cartan condition on $\eta$ and $\phi$ together with the Leibniz rule ensures \mbox{$(d_{\eta, \phi})^2 = 0$}.

\begin{example}\label{eg-dgpath}
For example $\cat C_1$ agrees with the path object in $\dgCat$ as constructed in section 3 of \cite{Tabuada10a}.
Indeed, objects are homotopy invertible morphisms $\eta\colon A \to B$ and morphisms from $\eta$ to $\phi$ are given by triples $(a_0, a_1, a_{01})$ with differential 
\[\de\colon (a_0, a_1, a_{01}) \mapsto (da_0, da_1, da_{01} + \phi \circ a_0 - (-1)^{|a_1|} a_1 \circ \eta)\]
\end{example}

Note that there are induced face and degeneracy maps.
The maps in the simplex category induce restriction functors $\partial_{i}\colon \cat C_n \to \cat C_{n-1}$ and inclusions $\sigma_{i}\colon \cat C_{n} \to \cat C_{n+1}$ that add an extra copy of $E_{i}$, connected by the identity map to $E_{i}$. We then define the maps $\eta$ by pullback, with the extra rule that $(\sigma_j(\eta))_{i_0, …, i_m} = 0$ if $\sigma_{j}(i_k) = 
\sigma_{j}(i_{m})$ for some $k \neq m$.

The replacement map $\iota: c\cat C \to \cat C_n$ is given by $(\sigma_0)^n$ in degree $n$.

Before we embark on the somewhat technical proof that $\cat C_{\bullet}$ is a simplicial resolution, we note the following application. We can extend the definitions of the differentials and composition to functions defined on general simplices. (That is, we replace ``leaving out the $i$-th term'' by the map induced by $\partial_{i}$ etc.)

\begin{proposition}\label{propn-ils-explicit}
Given a simplicial set $K$ we can construct $\cat C^{K}$ as the dg-category with objects $(E, \eta)$ where $E \in (\ob \cat C)^{K_{0}}$ and $\eta$ assigns to every $k$-simplex in $K_{\geq 1}$ a map in $\uHom_{k-1}(E({\partial_{0}^{k}\sigma}), E(\partial_{max}^{k}\sigma))$ satisfying the Maurer--Cartan equations. Hom-spaces are defined similarly to hom-spaces in $\cat C_{\bullet}$.
\end{proposition}
\begin{proof}
This follows from the  construction of $\cat C^{K} = \lim_{\De K\op} \cat C_{\bullet}$. All the copies of $\cat C_{n}$ corresponding to degenerate simplices are themselves degenerate. 
\end{proof}

\begin{remark}\label{rk-others}
Note that this shows that the construction of Morita cohomology in \cite{Holstein1} as $K \mapsto \cat C^{K}$ corresponds to $\oo$-local systems as defined in \cite{Block09}. \end{remark}
\begin{notation} Given an object or morphism $\al$ and a positive integer $k$ we write $\al_{[k]}$ for the collection of all $\al_{i_0 \dots i_k}$.
\end{notation}
\begin{proposition} \label{propn-mcquasiequiv}
The inclusion from the constant simplicial dg-category $c\cat C$ to $\cat C_{\bullet}$ is a levelwise weak equivalence. 
\end{proposition}
\begin{proof} We have to check that the inclusion map $\iota\colon c\cat C \to \cat C_n$ is a quasi-equivalence.

Let us first show that $\iota$ induces weak equivalences on hom-complexes.
We have to show that $\uHom_{\cat C^{\De ^n}}((\underline E, \eta), (\underline F, \phi)) \simeq \uHom_{\cat C}(E, F)$ when both $\eta$ and $\phi$ are of the form $(\id, 0)$, i.e.\ the constituent morphisms in degree 0 are the identity and all others are 0. 

Write $(H, d_H) \coloneqq \uHom(E, F)$ and note that from the definitions we can write
\[
\uHom((E, 0), (F, 0)) \simeq (H[1] \otimes \bigwedge \langle e_0, \dots, e_n\rangle, D)
\]
Here the $e_i$ all have degree 1 and we identify $H.e_{i_{0}} \wedge \dots \wedge e_{i_{k}}$ with the $a(i_{0}, \dots, i_{k})$.
The differential $D$ is $d_H + \iota_{\sum e_i}$ where the second term denotes contraction. This complex is a resolution of $(H, d_H)$.

Next we show $\iota$ is quasi-essentially surjective, i.e.\ show that any object $(E, \eta)$ is equivalent to an object $(\underline F_{0}, (\id, 0))$ where $\underline F_{0}$ is of the form $(F_{0}, \dots, F_{0})$.

We can deduce this if we can show that every $(E, \eta)$ is equivalent to some $(F, \phi)$ such that all compositions which agree up to homotopy by $\de \phi+ \phi^2 = 0$ agree strictly, i.e.\ $\phi = (\phi_{[0]}, 0)$, and that any such $(F, (\phi_{[0]}, 0))$ is equivalent to $(\underline F_{0}, (\id, 0))$. 
The second part of this is immediate: We define a map from $(\underline F_0, (\id, 0))$ to $(F, (\phi_{[0]}, 0))$ by sending $F_0$ to $F_i$ via $\phi(0,i) = \phi(i-1, i) \cdots \phi(0,1)$. Since all $\phi({j,j+1})$ are homotopy invertible there is a homotopy inverse. 

We will now show that any $(E, \eta)$ is equivalent to $(F, \phi)$ where $\phi$ has no higher homotopies. Let $F = E$ and let $\phi(i, j) = \eta(j-1, j) \cdots \eta(i, i+1)$.
We may assume by induction on $n$ that all $\eta(i_{0}, \dots i_{k})$ with $i_{k} < n$ are 0. 
Let us now factor the map from $(E, \eta)$ to $(E, \phi)$ as $(E, \eta) \to (E, \theta) \to (E, \phi)$ where $\theta$ is defined like $\phi$ on indices not including n and like $\eta$ otherwise. We first show the first map is a homotopy equivalence. By induction assumption we know this holds for $n-1$. So there is a homotopy equivalence $H’$ between the restrictions of $(E, \eta)$ and $(E, \theta)$ to the index set $0, \cdots, n-1$. We now extend this to homotopy equivalence $H$ by defining $H(n) = \id$ and $H(i_0, \dots, i_{k}, n) = 0$. This still has a homotopy inverse, defined in the same way but starting with the homotopy inverse of $H’$. Moreover $dH = dH’ = 0$.

Now we show the second map is an equivalence as well.
We define the homotopy equivalence $H\colon (E, \eta) \to (E, \phi)$ as follows:
\begin{align*}
H(i) &= \id \\
H(i_{0}, \dots, i_{k}) &= (-1)^{k-1}\eta(i_{0}, \dots, i_{k-1}, n-1, n) & \textrm{ if } i_{k} = n \textrm{ and }i_{n-1} \neq n-1 \\
H(i_{0}, \dots, i_{k}) &=0 & \textrm{ otherwise}
\end{align*}

And define $H^{-}$ to be equal to $H$ in degree $0$ and $-H$ in degree $> 0$. 

Then it is clear that $H$ and $H^{-}$ are inverses. Since $H(i_{0}, \dots, i_{n})$ is zero unless $i_{n} = n$ there are no nontrivial compositions and the compositions $\id \circ H(\dots)$ and $H^{-}(\dots) \circ \id$ cancel in degrees greater than 0.

So it remains to show that $dH = dH^{-} = 0$ to show we have a genuine homotopy equivalence.

We consider $H$ first.

Putting together our definitions we find the following. Let us first assume $i_{k-1} \neq n-1$ and $i_{k} = n$. To obtain the correct signs recall that $|H| = 0$ and $|\eta| = |\phi| = 1$.
\begin{align*}
(dH)(i_{0}, \dots, i_{k}) &= d(H(i_{0}, \dots, i_{k})) + \sum_{j=1}^{k-1} (-1)^{j}H(i_{0}, \dots, \hat i_{j}, \dots, i_{n}) \\
&  + \ \
\sum_{j=0}^{k} (-1)^{j} \phi(i_{j} \dots i_{n}) \circ H(i_{0}, \dots, i_{j})  - \sum_{j=0}^{k} H(i_{j}, \dots, i_{k}) \circ \eta(i_{0}, \dots, i_{j}) \\
\end{align*}
This simplifies to:
\begin{align*}
(dH)(i_{0}, \dots, i_{k})
& = (-1)^{k-1}d\eta(i_{0}, \dots, n-1, n) \\
& \  \ + (-1)^{k-2}\sum_{j} (-1)^{j}\eta(i_{0}, \dots, \hat i_{j}, \dots, n-1, n) \\
& \  \ + 0  
- (-1)^{k-2}\eta(i_{1}, \dots, n-1, n) \circ \eta(i_{0}, i_{1}) - \id \circ \eta(i_{0}, \dots, i_{k})
\\
&= 0
\end{align*}
The last equality holds since the penultimate term is of the form 
\[(-1)^{k-1}(\delta \eta + \eta^{2})(i_{0}, \dots, i_{k-1}, n-1, n)\]
This becomes clear if we write $\eta(i_{0}, \dots, i_{k}) = \eta(i_{0}, \dots, \widehat{n-1}, n)$ and observe that all the other terms we expect in $\delta \eta + \eta^{2}$ are 0.

The other cases are easier. If $i_{k} \neq n$ all terms in the differential are $0$ and if $i_{k-1} = n-1$ and $i_{k} = n$ there are only two nonzero terms, which cancel.

When we consider $dH^{-}$ the sign of the term $\eta(i_{0}, \dots, i_{k})$ changes, as it now comes from $\eta \circ H$ and not $H \circ \eta$. This cancels the effect of the sign of $H(i)$ also changing by a factor of $-1$. There are no other occurrences of the sign of $H(i)$ unless $k=1$ when all but the last two terms are zero and the last two terms cancel. 
\end{proof}

\begin{proposition}\label{prop-reedyfibrant}
The simplicial dg-category $\cat C_{\bullet}$ is Reedy fibrant.
\end{proposition}
\begin{proof}
Write 
\[\eta_{< n} \coloneqq (\eta_0, \dots, \widehat{\eta_{0 \dots n}}) = (\eta_{[0]}, \dots, \eta_{[n-1]})\]
Then $M_{n}(\cat C)$ is a subcategory of $\cat C_n$ whose objects are of the form $(E, \eta_{<n})$. In particular note that the Maurer--Cartan condition holds on all indexing sets except on $(0, \dots, n)$.
Similarly, morphisms are of the form $s_{<n}$ where $s$ is a morphism in $\cat C_n$.
This is easily seen to be the correct limit, see Proposition \ref{propn-ils-explicit}. We write $\pi \colon \cat C_{n} \to M_{n}\cat C$ for the functor forgetting $\eta_{[n]}$.

It is immediate from the definition that there is a surjection on hom-spaces. So it remains to check the lifting property for homotopy invertible maps. 
We will first reduce to lifting contractions, as is done in the case of path objects in section 3 of \cite{Tabuada10a}. 

Note that by assumption the dg-category $\cat C$ is fibrant and hence has cones,\footnote{This is incorrect, see the appendix.}
 cf. section 2 of \cite{Tabuada10a}.
 Then to see if a map $h$ is homotopy invertible it suffices to check that $cone(h)$ is contractible. 

So assume $h\colon (E, \eta_{<n}) \to (F, \phi_{<n})$ is homotopy invertible in $M_{n}\cat C$ with homotopy inverse $g$ and that $(E, \eta_{<n})$ is in the image of $\cat C_{n}$ under $\pi$. First we need to check that $(F, \phi_{<n})$ is also in the image of $\cat C_{n}$.
It is enough to find $\phi_{[n]}$  such that $\de\phi + \phi^{2} = 0$ while we know that $\de \phi_{<n} + \phi_{<n}^{2} = 0$. 
In other words we are looking for $\phi_{[n]}$ such that 
$d\phi_{[n]} = (\De \phi + \phi^2)_{[n]}$.

We will first consider $g(n) \cdot (\De \phi + \phi^2)(0 \dots n)$. Define $\rho \circ' \sigma$ to be $\rho \circ \sigma$ minus the term $\rho(n)\cdot \sigma(\cdots n)$. Then $- \circ' \sigma = - \circ \sigma$ if $\sigma$ is $\eta$ or $\phi$. Note that $d$ and $\De$ are compatible with $\circ'$ just as with the usual product.

Then $g(n) \cdot \phi(i \dots n) = (- g \circ' \phi + \eta \circ g - \de g)(i \dots n)$ and we can perform the following computation, where we deduce the Maurer--Cartan condition in degree $n$ from the Maurer--Cartan conditions in lower degrees.  
\begin{align*}
g(n)\cdot (\De \phi + \phi \circ \phi) 
&= - \De (g \circ' \phi) + \De (\eta \circ g) - \De \de g \\
&  \ \ + (- g \circ' \phi + \eta \circ g - \de g) \circ \phi \\
&= - g \circ' (\De \phi + \phi \circ \phi) + \eta \circ (\De g + g \circ \phi)  \\
&  \ \ - d g \circ' \phi + \De \eta \circ g + d \De g \\
&\simeq g \circ' d\phi - dg \circ' \phi + \eta  \circ \eta \circ g- \eta \circ dg  +  \De \eta \circ g\\
&= d(g \circ' \phi) - d\eta \circ g - \eta \circ dg  \\ 
&\simeq - d(\eta \circ g) \\
&\simeq 0
\end{align*}

Since $dh(n) = 0$ we deduce that $h(n)g(n)(\De \phi + \phi^{2}) \simeq 0$ and it suffices to show $(h(n)g(n) -
\id) \cdot (\De \phi + \phi^{2}) \simeq 0$. We know there exists $K$ with $dK = h(n)g(n) - \id$ so the desired homotopy follows if we can show that $d(\De\phi + \phi^{2}) = 0$.
One may check explicitly that $d(\De \phi) =  - \De \phi \circ \phi + \phi \circ \De \phi$, using the fact that $d\phi = - \De \phi - \phi^{2}$ in degree less than $n$. 
Then we can use Maurer--Cartan in lower degrees again to deduce:
\begin{align*}
d(\De \phi + \phi^{2}) &= d(\De \phi)- (-\De \phi - \phi^{2}) \circ \phi + \phi \circ (- \De \phi - \phi^{2})\\
&= 0
\end{align*}

Thus we know the domain and codomain of $h$ are in the image of $\pi$ and we can use surjectivity of hom-spaces to write $h = \pi(H)$. Now it suffices to show that the contraction of $h$ lifts. 

Let us assume we are given a contraction $s_{< n}$ of $cone(h) = (G, \gamma_{< n})$,
we have to find a contraction $s$ of $(G, \gamma)$.\footnote{As $H$ need not be closed this need not exist, see the appendix.} By assumption we can write $d_\gamma(s_{<n}) = (\id, 0, \dots, 0, t_{[n]})$ for some $t_{[n]}$. Now consider $0 = d_\gamma d_\gamma(s_{<n}) = (0, \dots, 0, dt_{[n]} + 0)$. This forces $\de t_{[n]} = d t_{[n]}= 0$. But now we know that $d s_{[0]} = \id$ and hence $d\colon s_{[0]} t_{[n]} \mapsto t_{[n]}$ and $(s_{[0]}, \dots, s_{[n-1]}, s_{[0]} t_{[n]})$ is a contraction of $(G, \gamma)$.

We deduce that $H$ is contractible and the preimages of $(E, \eta_{<n})$ and $(F, \phi_{<n})$ are indeed homotopy equivalent.
\end{proof}

\appendix
\section{Errata}

As pointed out to me by Daria Poliakova there are gaps in the proof that  $\cat C_\bullet$ is indeed a simplicial resolution of a dg category $\cat C$.

To discuss this, we need to distinguish between the Dwyer-Kan and Morita model structures.
Let us first fix the Dwyer-Kan model structure on $\dgCat_k$ und fix a dg category $\cat C$.

 The argument given in the body of this paper has two errors:
 Firstly, we assume $\cat C$ has cones in Proposition \ref{prop-reedyfibrant}, which follows if $\cat C$ is Morita fibrant, but we are working in the DK model structure and may not assume this.
 One can work in the pretriangulated hull of $\cat C$ to overcome this.
More seriously, in the proof of the same proposition we need to show we can lift a homotopy invertible map $h$. To do this we consider the cone of a lift $H$ with $\pi(H)=h$. This approach falls short as $H$ as constructed is not necessarily closed. 
(It seems that this issue also occurs when the path object for dg-categories is considered in Proposition 3.3 of \cite{Tabuada10a}, cf.\ also \cite{Shoikhet18}.)

More careful explicit proofs are given by Arhkipov and Poliakova, see Theorem 3.36 in \cite{Arkhipov18}, which shows that our construction in Definition \ref{def-construction} gives simplicial resolutions in the Dwyer Kan model structure.

Next we consider the case of the Morita model structure. The issue in this case is that it is a priori harder to check that the maps $\cat C_n \to M_n \cat C$ are Morita fibrations.

We note that the Morita and Dwyer-Kan model structures on $\dgCat_k$ induce two different Reedy model structures on simplicial dg categories.
The construction in Definition \ref{def-construction} applied to Morita fibrant (i.e.\ pre-triangulated) dg category provides a level-wise Morita fibrant dg category $\cat C_\bullet$ that is weakly equivalent to $c\cat C$ in either Reedy model structure.
It is moreover Dwyer-Kan Reedy fibrant using our main results with corrections from \cite{Arkhipov18}. It remains to show that $\cat C_\bullet$ is also Morita Reedy fibrant. But that is entirely formal.

In fact, the following argument applies to any Bousfield localization of model categories.

\begin{lemma}
	Let $X_\bullet$ be a simplicial dg category that is Dwyer-Kan Reedy fibrant and level-wise Morita fibrant. Then $X_\bullet$ is Morita Reedy fibrant.
\end{lemma}

\begin{proof}
We take $A_\bullet \to B_\bullet$ a Morita Reedy acyclic cofibration and consider a diagram
$$\xymatrix{
		A_\bullet\ar[r]\ar[d]_i & X_\bullet \ar[d] \\
	B_\bullet \ar[r]& 0
}$$
Taking level-wise pretriangulated envelopes gives a levelwise Morita fibrant replacement $RA_\bullet \to RB_\bullet$. As all objects are now Morita fibrant this is a level-wise weak equivalence in either model structure. 

As $X_\bullet$ is levelwise Morita fibrant our diagram factors through $RA_\bullet \to RB_\bullet$ as the pretriangulated envelope is left adjoint to inclusion.

Now we may factor $RA_\bullet \to RB_\bullet$ as an acyclic cofibration followed by an acyclic fibration in the the Morita Reedy model category, write this $RA_\bullet \to Y_\bullet \to RB_\bullet$. Note $Y_\bullet$ is levelwise Morita fibrant, so $j: RA_\bullet \to Y_\bullet$ is a level-wise weak equivalence in the Dwyer-Kan model structure as well, and thus $j$ is an
acyclic cofibration in the Dwyer-Kan Reedy model structure.

$$\xymatrix{
	A_\bullet\ar[r]\ar[dd]_i & RA_\bullet \ar[r]\ar[d]_j & X_\bullet \ar[dd] \\
	 & Y_\bullet \ar[d]_p \ar@{-->}[ur] & \\
	B_\bullet \ar[r]\ar@{-->}[ur] & RB_\bullet \ar[r] & 0
}$$

To show $X_\bullet$ is Reedy Morita fibrant we need to find a lift $B_\bullet \to X_\bullet$ in our original diagram.  We first find the lift $B_\bullet \to Y_\bullet$ as $i$ is an acyclic cofibration and $p$ is a fibration in the Morita Reedy model category. 

Then we find the lift $Y_\bullet \to X_\bullet$ as $j$ is an acyclic cofibration and $X_\bullet \to 0$ a fibration in the Dwyer-Kan Reedy model structure.
\end{proof}

\bibliography{../biblibrary}

\end{document}